\documentclass[a4paper,11pt]{article}
\usepackage[utf8]{inputenc}
\usepackage{amsmath,amsfonts,amssymb,anysize,amscd}
\usepackage[all,cmtip]{xy}
\usepackage[onehalfspacing]{setspace}
\usepackage{amsthm}

\theoremstyle{plain}
\newtheorem{thm}{Theorem}[section]

\newtheorem{lem}[thm]{Lemma}
\newtheorem{prop}[thm]{Proposition}

\theoremstyle{definition}
\newtheorem{defn}[thm]{Definition}

\theoremstyle{remark}

\font\smallsc=cmcsc10
\font\smallsl=cmsl10

\def\p{\mathbb{P}}
\def\z{\mathbb{Z}}

\def\L{\mathcal{L}}
\def\X{\mathcal{X}}
\def\O{\mathcal{O}}
\def\ox{\otimes}
\def\wt{\widetilde}
\def\:{\colon}
\def\M{\mathcal{M}}
\def\I{\mathcal{I}}
\def\N{\mathcal{N}}
\def\V{\mathcal{V}}
\def\ol{\overline}

\def\P{\mathcal{P}}
\def\W{\mathcal{W}}
\def\rk{{\rm rk}}
\def\Ker{{\rm Ker}}
\def\Im{{\rm Im}}
\def\g{{\mathfrak g}}
\def\IP{{\mathbb P}}
\def\div{\mathrm{div}}
\def\lra{\longrightarrow}
\title{Level-$\delta$ limit linear series}
\author{
Eduardo Esteves\footnote{Supported by CNPq, Processos~304259/2010-1
  and 304623/2015-6, 
by FAPERJ, Processo E-26/102.979/2011, and by
CAPES, Processo 4155-13-0.}, 
Antonio Nigro\footnote{Supported by CNPq, Processo 152063/2010-2.} and 
Pedro Rizzo\footnote{Supported by a doctor scholarship from
  CAPES-PROEX.}
}
\date{}

\begin{document}

\maketitle

\begin{abstract}
We introduce the notion of \emph{level-$\delta$ limit linear series},
which describe limits of linear series along families of smooth
curves degenerating to a singular curve $X$. We treat here only the
simplest case where $X$ is the union of two smooth components meeting
transversely at a point $P$. The integer $\delta$ stands for the
singularity degree of the total space of the degeneration at $P$. If
the total space is regular, we get level-1 limit linear series, which
are precisely those introduced by Osserman \cite{Oss1}. We construct a
projective moduli space $G^r_{d,\delta}(X)$ parameterizing
level-$\delta$ limit linear series of rank $r$ and degree $d$ on $X$,
and show that it is a new compactification, for each $\delta$, of the moduli space of
Osserman exact limit linear series, an open subscheme 
$G^{r,*}_{d,1}(X)$ of the space $G^r_{d,1}(X)$ already constructed by
Osserman. Finally, we generalize \cite{E-Oss} by associating 
to each exact level-$\delta$ limit linear series $\mathfrak
g$ on $X$ a closed subscheme $\p(\mathfrak g)\subseteq X^{(d)}$ of the $d$th
symmetric product of $X$, and showing that $\p(\mathfrak g)$ is the limit of
the spaces of divisors associated to linear series on smooth curves
degenerating to $\mathfrak g$ on $X$, if such degenerations exist. In
particular, we describe completely limits of divisors along
degenerations to such a
curve $X$.
\end{abstract}

\section{Introduction.}
The theory of linear series, meaning spaces of sections of line
bundles, 
has a long history and plays an important 
role in Algebraic Geometry. The special case of curves is particularly 
rich and has been investigated from several directions. Following on 
their proof of the Brill--Noether
Theorem in the 1980's, Eisenbud and Harris \cite{EH1} introduced the
notion of \emph{limit linear series}, providing a powerful framework
to study degenerations of linear series on families of smooth curves
degenerating to curves of compact type. 
As a consequence of their general theory, they were able to simplify
the proof of the Brill-Noether Theorem and to prove a number of other 
results about curves; see the introduction to \cite{EH1}.

The success achieved by Eisenbud and Harris in the 80's, and further
applications of the theory of limit linear series, particularly in
computing divisors in the moduli space of stable curves, has motivated
further study into the foundational aspects of the theory. For
starters it was observed that most of the applications of the theory
are obtained by considering an open subscheme of 
the projective moduli space $G^r_d(X)$, 
parameterizing limit linear 
series of rank $r$ and degree $d$ 
on a curve of compact type $X$, that whose points correspond to
special limit linear series called \emph{refined}. Not always limit
linear series on a curve $X$ are limits of linear series on smooth
curves degenerating to $X$. In any case, refined limit $g^1_d$ are always limits of 
linear series, and for $g^r_d$ with $r\geq 2$ a 
similar result, called Regeneration Theorem, holds; see \cite{EH1},
Thm.~3.4, p.~360. This is Eisenbud's and Harris' main theorem, the one often
used in applications. Unfortunately, there are points on $G^r_d(X)$
that are not refined, but carry important information, as they correspond
to limit linear series that are actually limits of linear series.

A fundamental breakthrough was made by Osserman \cite{Oss1} in the
2000's, 
when he gave a new definition of limit linear 
series on a curve $X$ and constructed a projective moduli space
$G^{r,\text{Oss}}_d(X)$ parameterizing those new objects. An Osserman
limit linear series carries more information than a usual one, and
thus there is a forgetful map $G^{r,\text{Oss}}_d(X)\to G^r_d(X)$. The
map is an isomorphism over the refined locus of $G^r_d(X)$, and thus
$G^{r,\text{Oss}}_d(X)$ can be viewed as a different natural
compactification of this locus; see \cite{Oss1}, Section~6, p.~1183. 
Among the Osserman limit linear series there are those called
\emph{exact}, which form an open subscheme of
$G^{r,\text{Oss}}_d(X)$. Exact limit linear series are more amenable
to work. And they are dense among all limit linear series, by
\cite{Liu}, Cor.~1.4, p.~4034, at least if $X$ is general. 
Furthermore, all limits of linear series along families
of smooth curves degenerating to $X$ are exact, \emph{as long as} the
total space of the family is regular; see \cite{E-Oss}, Section~5,
p.~90. The space $G^{r,\text{Oss}}_d(X)$  has a defect similar to that
of $G^r_d(X)$ tough:
there are non-exact limit linear series that are limits of linear
series, along a nonregular smoothing of $X$.

Exact limit linear series are also special in the following sense: A
linear series $\mathfrak g$ of rank $r$ and degree $d$ on a smooth curve $C$
corresponds to a subscheme $\p(\mathfrak{g})$ of the symmetric product
$C^{(d)}$ of $d$ copies of $C$, whose points parameterize the
divisors of zeros of sections of $\mathfrak g$. This correspondence is
fundamental in the theory of curves. Given a family of linear series $\mathfrak g$ of
degree $d$ on smooth curves degenerating to a singular curve $X$, we
may ask what the limit of $\p(\mathfrak{g})$ in $X^{(d)}$ is. In \cite{E-Oss} Osserman and the 
first author considered the corresponding subscheme $\p(\mathfrak{g})$
of $X^{(d)}$ associated to an Osserman limit linear series $\mathfrak
g$ on $X$. It is defined as for smooth curves but, as certain 
sections of $\mathfrak g$ may vanish on a whole component of $X$, the
subscheme $\p(\mathfrak{g})$ is actually the closure of the locus of
divisors of zeros of the other sections. The two showed then that, if
$\mathfrak g$ is exact, then $\p(\mathfrak{g})$ has the expected
Hilbert polynomial, as if $\mathfrak g$ were a limit, and in this case
$\p(\mathfrak{g})$ is actually the limit of the schemes corresponding
to the linear series degenerating to $\mathfrak g$; see \cite{E-Oss},
Thms.~4.3 and 5.2. The converse is
shown here, our Theorem \ref{thm:1}.

Our goal in the present paper is to study limits of linear series
along nonregular smoothings of a given $X$, that is, along families of
smooth curves degenerating to $X$ whose total space is not regular. In
principle, as $X$ is nodal, one could replace the family by its
semistable reduction, thus replacing $X$ by a curve $\widetilde X$
obtained from $X$ by replacing the nodes by chains of rational
curves. However, dealing with $\widetilde X$ instead of $X$ is substantially
more difficult in the approach by Osserman. So much difficult that,
though Eisenbud and Harris developed their theory for curves of
compact type, only very recently (\cite{Oss1405} and \cite{Oss1406})
has Osserman extended his theory to curves $X$ that are not simply unions of
two smooth components meeting transversely at a single
point $P$. That is to say: From the above third paragraph of this
introduction on, and throughout the whole article, $X$ stands actually for such a
simple curve!

So we take a different approach: We introduce what we call
{\em level-$\delta$ limit linear series} on $X$; see Definition~\ref{def:1}. As in
Osserman's work, there are special level-$\delta$ limit linear series,
also called {\em exact}. And we show that certain exact level-$\delta$ limit linear
series arise as limits of linear series along smoothings of $X$ whose
total space has singularity degree $\delta$ at $P$; see the discussion
before Definition~\ref{def:1}. Level-1 limit linear series are simply Osserman
limit linear series. Following Osserman, we construct a projective
moduli space $G^r_{d,\delta}(X)$ for level-$\delta$ limit linear series
in Proposition~\ref{p1}. 

Level-$\delta$ limit linear series carry more information than
Osserman limit linear series. There is in fact a forgetful map 
$G^r_{d,\delta}(X)\to G^{r,\text{Oss}}_d(X)$. More generally, there are
forgetful maps $\rho_{\delta',\delta}\colon G^r_{d,\delta'}(X)\to
G^r_{d,\delta}(X)$ as long as $\delta|\delta'$. In our first main result, Theorem~\ref{p2},
we show that $\rho_{\delta',\delta} $ is surjective and describe its
fibers. As a consequence, we show in
Proposition~\ref{Cor1} that $\rho_{\delta',\delta}$ is an isomorphism
over the open subscheme $G^{r,*}_{d,\delta}(X)$ of $G^r_{d,\delta}(X)$
parameterizing exact limit linear series. 
Also, $\rho_{\delta',\delta}^{-1}(G^{r,*}_{d,\delta}(X))\subseteq
G^{r,*}_{d,\delta'}(X)$ and
$\rho_{\delta',\delta}(G^{r,*}_{d,\delta'}(X))=G^r_{d,\delta}(X)$ if $\delta'>\delta$.
It turns out that, for each
$\delta$, we may view $G^r_{d,\delta}(X)$ as a compactification of the
moduli space of Osserman {\em exact} limit linear series.

Finally, following \cite{E-Oss}, we associate to each level-$\delta$
limit linear series $\mathfrak g$ a subscheme 
$\mathbb P(\mathfrak g)$ of $X^{(d)}$. As in the level-1 case, we show
that if $\mathfrak g$ is exact, then $\p(\mathfrak{g})$ has the expected
Hilbert polynomial, as if $\mathfrak g$ were a limit, and in this case
$\p(\mathfrak{g})$ is actually the limit of the schemes corresponding
to the linear series degenerating to $\mathfrak g$. This is contained
in our last
two main results, Theorems~\ref{pv} and \ref{thm:1}, where the
converse is proved. The key to showing the converse is to show that if 
$\mathfrak g=\rho_{\delta',\delta}(\mathfrak g')$, then $\p(\mathfrak
g)\subseteq\p(\mathfrak g')$. Furthermore, equality holds if $\mathfrak g$ is exact
and does not hold if $\mathfrak g$ is not exact but $\mathfrak g'$ is.

In a forthcoming article \cite{ENR2}, we will give yet another notion
of limit linear series on $X$, and construct a moduli space
which will be a sort of glueing of the exact loci
$G_{d,\delta}^{r,*}(X)$ for all $\delta$, modulo a certain equivalence
relation. The remarkable fact is that this new moduli space is
projective, thus giving rise to a compactification of the locus of
Osserman exact limit linear series, and thus of Eisenbud and Harris
refined limit linear series by exact limit linear series, precisely
those which have good properties, as for instance those found in the
present article.

Part of this work appeared in the third author's doctor thesis at IMPA
\cite{Rizzo}. We would like to thank Margarida Melo, Brian Osserman,
Marco Pacini and Filippo Viviani for helpful discussions on the subject.

\section{Twists}
\label{Tw}

Let $X$ be a projective curve defined over an algebraically closed field
$k$. Assume that $X$ has exactly two irreducible 
components, denoted $Y$ and $Z$, that they are smooth and intersect transversally 
at a single point, denoted $P$. Let $\pi\colon\mathcal X\to B$ be a 
{\it smoothing} of $X$, that is, a flat, projective map
to $B:=\text{Spec}(k[[t]])$ whose generic fiber is smooth and
special fiber is isomorphic to $X$. We let $\eta$ and $o$ denote the
generic and special points of $B$, and $X_\eta$ and $X_o$ the
respective fibers of $\pi$. Notice that, by semicontinuity, not only
is $X_\eta$ smooth, but also geometrically connected. We
will fix an identification of $X_o$ with $X$.

Since $\pi$ is flat and $B$ is regular, the total space $\mathcal X$ is
regular except possibly at the node $P$. Furthermore, since the general
fiber of $\pi$ is smooth,
there are a positive integer $\delta$ and a $k[[t]]$-algebra
isomorphism (see \cite{ACH}, pp.~104--109):
\begin{equation}
\widehat{\mathcal{O}}_{\mathcal{X},P}\cong\frac{k[[t,y,z]]}{(yz-t^{\delta})}.\label{lociso}
\end{equation}
\noindent
The integer $\delta$ is called the {\it singularity degree of
$\pi$ at $P$}. (Also, we say that the singularity of
$\mathcal{X}$ at $P$ is of type $A_{\delta-1}$.) We say that $\pi$ is
a {\it regular smoothing} if its singularity type at $P$
is $1$, in other words, if $\mathcal{X}$ is regular.

\begin{lem}\label{lemma1} There is a unique effective Cartier divisor
of $\X$ whose associated $1$-cycle is $i[Y]$ (resp.~$i[Z]$) if and only
if $\delta|i$.
\end{lem}

\begin{proof} Since $\X$ is regular off $P$, there is a unique 
effective Cartier divisor $\mathcal Y_i^*$ on $\X^*:=\X-P$ whose associated $1$-cycle
  is $i[Y-P]$. If there were an effective Cartier divisor on $\X$
  associated to $i[Y]$, it would be the schematic closure $\mathcal
  Y_i$ of $\mathcal Y_i^*$, whence unique. 

Now, fix an isomorphism of the form
  (\ref{lociso}), and let $A:=k[[t,y,z]]/(yz-t^\delta)$. Up to
  exchanging $y$ with $z$, the ideal
  defining $Y$ (resp.~$Z$) in $\widehat{\mathcal{O}}_{\mathcal{X},P}$ corresponds
  to $(y,t)A$ (resp.~$(z,t)A$) in $A$. Let $I\subset A$ be the ideal
  corresponding to $\mathcal Y_i$. Localizing, $I_z=t^{i}A_z$, and thus
$I=I_z\cap A$.

We claim that $I_z\cap A=y^q(y,t^r)A$, where $q$ is the
quotient in the Euclidean division of $i$ by $\delta$ and $r$ is the
remainder. Indeed, since
$yA_z=t^\delta A_z$, it follows that $y^q(y,t^r)A_z=t^iA_z$, and thus
$I_z\cap A\supseteq y^q(y,t^r)$. On the other
hand, if $g\in I_z \cap A$ then there is an integer $n\geq i$ such
that $z^ng\in t^iA$. Thus $z^ng\in z^qy^qt^rA$. Since $y,z$ form a
regular sequence of $A$, it follows that $g=y^qg'$, where
$z^{n-q}g'\in t^rA$. Since $z$ is not a zero divisor modulo
$(y,t^r)A$, it follows that $g'\in (y,t^r)A$, and thus $g\in y^q(y,t^r)A$.

Finally, since $I=y^q(y,t^r)A$, we have that $I$ is principal if and only if
$(y,t^r)A$ is principal, thus if and only if $r=0$.
\end{proof}

We let $\delta Y$ (resp.~$\delta Z$) denote the effective Cartier
divisor of $\mathcal {X}$ whose associated $1$-cycle is $\delta[Y]$
(resp.~$\delta[Z]$).

\begin{prop}\label{Rem3.1} The following statements hold:
\begin{description}
\item (a) $\delta Y\cdot Z=\delta Z\cdot Y=1$.
\item (b) $\O_{\mathcal{X}}(\delta Y)|_Z\cong\O_Z(P)$ and
$\O_{\mathcal {X}}(\delta Z)|_Y\cong\O_Y(P)$.
\item (c) $\O_{\mathcal{X}}(\delta Y)|_Y\cong\O_Y(-P)$ and
$\O_{\mathcal {X}}(\delta Z)|_Z\cong\O_Z(-P)$.
\end{description}
\end{prop}

\begin{proof} Fixing an isomorphism of the form
  (\ref{lociso}), we have that $Y$ (resp. $Z$) is defined at $P$ by, say,
  $(y,t)$ (resp.~$(z,t)$), whereas $\delta Y$ (resp. $\delta Z$) is
  defined by $y$   (resp.~$z$). The first two statements follow.
As for the last statement, it is enough to observe that
$\delta Y+\delta Z=\text{div}(t^\delta)$, and thus $\O_\X(\delta
Y)\ox\O_\X(\delta Z)=\O_\X$.
\end{proof}

Blowing up $\X$ at $P$, and then successively at the singular points
of each blowup, we end up with a regular scheme $\wt\X$ and a map
$\psi\:\wt\X\to\X$ such that the composition $\wt\pi:=\pi\psi$ is a
regular smoothing of its special fiber. Furthermore, the special fiber
can be identified with the curve $\wt X$ obtained from $X$ by
splitting the branches of $X$
at the node $P$ and connecting them by a chain $E$ of rational smooth curves
of length $\delta-1$, in such a way that $\psi|_{\wt X}\:\wt X\to X$
is the map collapsing $E$ to $P$. We say that $\psi$ is the
\emph{semistable reduction} of $\pi$.

We will also denote by $Y$
(resp. $Z$) the irreducible component of $\wt X$ mapped isomorphically
by $\psi$
to $Y$ (resp.~$Z$) on $X$. Also, we identify the generic fiber of
$\wt\pi$ with that of
$\pi$ through $\psi$. Finally, we will order the rational components
$E_1,\dots,E_{\delta-1}$ in such a way that $E_1$ intersects $Y$,
while $E_{\delta-1}$ intersects $Z$, and $E_i$ intersects $E_{i+1}$
for $i=1,\dots,\delta-2$.

\begin{prop} Let $L_\eta$ be an invertible sheaf on $X_\eta$ of degree
  $d$. Then
  there is an invertible sheaf $\L$ on $\wt\X$ whose restriction to the
  generic fiber is $L_\eta$, whose restriction to $Y$ (resp.~$Z$)
  has degree $d$ (resp.~$0$) and whose restriction to $E$ is trivial.
\end{prop}

\begin{proof} Since $\wt\X$ is regular, there is an invertible
  extension $\M$ of $L_\eta$ to $\wt\X$. Let $m$ and $n$ be the
  degrees of its restriction to $Y$ and $Z$, respectively, and $d_i$
  the degree on $E_i$ for $i=1,\dots,\delta-1$. Since $\wt\X$ is
  regular, $Y$, $Z$ and the $E_i$ are effective Cartier divisors of
  $\wt X$. A simple computation shows that
$$
\L:=\M\ox
\O_{\wt\X}\Big(nZ+\sum_{i=1}^{\delta-1}(d_i+\cdots+d_{\delta-1}+n)(Z+E_i+\cdots+E_{\delta-1})
\Big)
$$
has the required degrees.
\end{proof}

Let $\I$ be a \emph{torsion-free rank-1} sheaf on $\X/B$. In other words,
$\I$ is a coherent sheaf on $\X$, flat over $B$, invertible everywhere
but possibly at $P$, and such that $\I_P$ is isomorphic to an ideal of
$\O_{\X,P}$. Any invertible sheaf on $\X$ is torsion-free rank-1 on
$\X/B$, for instance $\O_\X$.

In \cite{EE2}, \S 3, a procedure was outlined to modify
$\I$: Set $\mathcal{I}^{(0)}:=\I$, and for each integer $i>0$,
define the \emph{$i$-th twist by $Z$} of $\I$ by:
$$
\mathcal{I}^{(i)}:=\ker\left(\mathcal{I}^{(i-1)}\longrightarrow
\frac{\mathcal{I}^{(i-1)}|_{Z}}{\text{torsion}}
\right).
$$
Notice that $\I^{(i)}\supseteq\I^{(i-1)}\I_{Z|\X}$, with equality away
from $P$. In particular, the $\mathcal{I}^{(i)}$ are all equal away from
$Z$, thus on $X_\eta$. Furthermore, as pointed out in \cite{EE2},
p.~3063, it follows from an argument analogous to the one found
in \cite{La}, Prop.~6, p.~100, that the $\I^{(i)}$ are
torsion-free, rank-1 on $\X/B$. Finally, it is stated in
\cite{EE2}, Lemma 23, p.~3063, that there is a natural surjection
of short exact sequences:
\begin{equation}\label{propcel}
\begin{CD}
0 @>>> \I^{(i+1)} @>>> \I^{(i)} @>>>
\frac{\mathcal{I}^{(i)}|_{Z}}{\text{torsion}} @>>> 0\\
@. @VVV @VVV @| @.\\
0 @>>> \frac{\mathcal{I}^{(i+1)}|_{Y}}{\text{torsion}} @>>>
\mathcal{I}^{(i)}|_{X} @>>>
\frac{\mathcal{I}^{(i)}|_{Z}}{\text{torsion}} @>>> 0.
\end{CD}
\end{equation}

Clearly, $\O_\X^{(1)}=\I_{Z|\X}$. Also, $\O_\X^{(i)}$ is a
sheaf of ideals containing $\I^i_{Z|\X}$ and equal to it away from
$P$, for each $i\geq 0$. If $\I$ is invertible then,
by exactness of $\I\ox -$, we have that $\I^{(i)}=\I\ox\O_\X^{(i)}$ for
each $i\geq 0$. In this case, it follows from \cite{CEG}, Prop.~3.1,
p.~13, that there
are isomorphisms
\begin{eqnarray}\label{propcel2}
\frac{\mathcal{I}^{(i+1)}|_{Y}}{\text{torsion}}\cong\I|_Y\ox\O_{Y}(-(q+1)P)
\quad\text{and}\quad
\frac{\mathcal{I}^{(i)}|_{Z}}{\text{torsion}}\cong\I|_Z\ox\O_{Z}(qP),
\end{eqnarray}
where $q$ is the quotient of the Euclidean division of $i$ by
$\delta$. Furthermore, it follows from the local description given in
the first paragraph of the proof of \cite{CEG}, Prop.~3.1, p.~14,
that $\I^{(i)}$ is invertible if and only if $\delta|i$, in which case
$\mathcal{I}^{(i)}\cong\I\ox\O_{\mathcal{X}}(-q(\delta Z))$.

Thus we have a complete description of the sheaves
$\I^{(i)}|_X$ when $\I$ is invertible. Namely, for each integer $i$,
letting $q$ and $r$ be the quotient and the remainder of the Euclidean
division of $i$ by $\delta$, if $r=0$ then $\I^{(i)}|_X$ is the
invertible sheaf on $X$ whose restrictions to $Y$ and $Z$ are
$\I|_Y\ox\O_{Y}(-qP)$ and $\I|_Z\ox\O_{Z}(qP)$, unique since $X$ is of
compact type. On the other hand, if $r\neq 0$, then $\I^{(i)}|_X$
is not invertible, whence the bottom sequence in Diagram \eqref{propcel}
splits and we have $\I^{(i)}|_X=\I|_Y\ox\O_{Y}(-(q+1)P)\oplus
\I|_Z\ox\O_{Z}(qP)$.

There is a parallel construction on $\wt{\X}$, which is helpful to have
in mind. Namely, if $\L$ is an
invertible sheaf on $\wt\X$, let $\L^{(0)}:=\L$, and for each integer
$i>0$ let
$$
\L^{(i)}:=\L^{(i-1)}\otimes\O_{\wt\X}(-E_r-E_{r+1}-\cdots-E_{\delta-1}-Z),
$$
where $r$ is the remainder of the Euclidean division of $i$ by
$\delta$, if $r>0$, and
$$
\L^{(i)}:=\L^{(i-1)}\otimes\O_{\wt\X}(-Z)
$$
if $r=0$.

Let $\L$ be an invertible sheaf on $\wt\X$ whose restriction
$\L|_{E_j}$
has degree 0 but for
at most one $j$, for which the degree is 1. Let $\ell:=0$ if no such $j$ occurs; otherwise, let
$\ell$ be that $j$. A simple computation shows that $\L=\M^{(\ell)}$ for a
certain invertible sheaf $\M$ whose restriction to $E$ is
trivial. Furthermore, $\M^{(i)}|_{E_j}$ has
degree 0 for each $j\neq r$, whereas
$\M^{(i)}|_{E_r}$ has degree $1$, where $r$ is the remainder of the
Euclidean division
of $i$ by $\delta$. It follows from \cite{EP}, Thm.~3.1,
that $R^1\psi_*\M^{(i)}=0$, and $\psi_*\M^{(i)}$ is a
torsion-free, rank-1 sheaf on $\X/B$ whose formation commutes with
base change, for each $i\geq 0$. Furthermore, $\psi_*\M^{(i)}$ is
invertible if and only if $\delta|i$.

\begin{prop}\label{LI} If $\I=\psi_*\L$ then
  $\I^{(i)}=\psi_*\L^{(i)}$ for each $i\geq 0$.
\end{prop}

\begin{proof} Clearly, $\I^{(0)}=\psi_*\L^{(0)}$. Assume by induction
  that $\I^{(i-1)}=\psi_*\L^{(i-1)}$ for a certain $i>0$. Let
  $r$ be the remainder of the
Euclidean division
of $i+\ell$ by $\delta$.

Suppose first that $r>0$. Consider the natural exact sequence defining
$\L^{(i)}$:
$$
0\to\L^{(i)}\to \L^{(i-1)} \to
\L^{(i-1)}|_{E_r+\cdots+E_{\delta-1}+Z}\to 0.
$$
Since $R^1\psi_*\L^{(i)}=0$, applying $\psi_*$ to it we get another
natural short exact sequence:
$$
0\to\psi_*\L^{(i)}\to\psi_*(\L^{(i-1)})\to
\psi_*(\L^{(i-1)}|_{E_r+\cdots+E_{\delta-1}+Z})\to 0.
$$
It remains to show that
$\psi_*(\L^{(i-1)}|_{E_r+\cdots+E_{\delta-1}+Z})$ is $\I^{(i-1)}|_Z$
modulo torsion.

To simplify the notation, set $\N:=\L^{(i-1)}$ and
$Z_r:=E_r+\cdots+E_{\delta-1}+Z$. We need to prove that
$\psi_*(\N|_{Z_r})$ is $\psi_*(\N)|_Z$ modulo torsion. First observe that the 
degree of $\N|_{E_j}$ is $0$ for every $j\geq r$. Thus
$$
\psi_*(\N|_{\ol{Z_r-Z}}(-Q))=0=R^1\psi_*(\N|_{\ol{Z_r-Z}}(-Q))=0,
$$
where $Q$ is the intersection of $E_{\delta-1}$ with $Z$. By applying $\psi_*$ 
to the short exact sequence
$$
0\to\N|_{\ol{Z_r-Z}}(-Q)\to\N|_{Z_r}\to\N|_{Z}\to 0,
$$
and considering the associated long exact sequence, it follows that 
$\psi_*(\N|_{Z_r})\cong\psi_*(\N|_Z)$. In particular, $\psi_*(\N|_{Z_r})$ is an 
invertible sheaf on $Z$, isomorphic to $\N|_Z$.

Since $\psi(Z_r)=Z$, there is a natural map
$$
h\:\psi_*(\N)|_{Z}\to\psi_*(\N|_{Z_r}).
$$
Since $\psi_*(\N|_{Z_r})$ is invertible, the torsion is mapped to
zero. Since the source is a quotient of $\I^{(i-1)}$, and the target
is a quotient of
$\psi_*(\L^{(i-1)})$, and both are equal
by induction hypothesis, it follows that $h$ is surjective. Since both source 
and target are rank 1, the kernel of $h$ is torsion. So $h$ induces an 
isomorphism between $\psi_*(\N)|_Z$ modulo torsion and $\psi_*(\N|_{Z_r})$.

If $r=0$, the proof goes through as before, but simpler, with $Z$
replacing $Z_r$.
\end{proof}

Clearly, we may define in a similar way the \emph{$i$-th twist by $Y$}
of $\I$, for every $i\geq 0$. For each $i\geq 0$,
let $\I^{(-i)}$ be $t^{-i}$ times the $i$-th twist by $Y$ of $\I$. Since the 
first twist is a subsheaf of $\I$ containing $t\I$, it follows that
$\I\subseteq\I^{(-1)}\subseteq\I^{(-2)}\subseteq\cdots$. Furthermore, the
notation is justified, as it follows from \cite{EE2}, Lemma 23, p.~{3063} that
\begin{equation}\label{ZI}
(\I^{(i)})^{(j)}=\I^{(i+j)}\quad\text{for all $i,j\in\mathbb Z$}.
\end{equation}
Also, by the same lemma, there is a natural surjection
of short exact sequences:
\begin{equation}\label{propcel3}
\begin{CD}
0 @>>> \I^{(i-1)} @>t>> \I^{(i)} @>>>
\frac{\mathcal{I}^{(i)}|_{Y}}{\text{torsion}} @>>> 0\\
@. @VVV @VVV @| @.\\
0 @>>> \frac{\mathcal{I}^{(i-1)}|_{Z}}{\text{torsion}} @>>>
\mathcal{I}^{(i)}|_{X} @>>>
\frac{\mathcal{I}^{(i)}|_{Y}}{\text{torsion}} @>>> 0.
\end{CD}
\end{equation}
Observe that, because of \eqref{ZI}, Diagrams \eqref{propcel} and
\eqref{propcel3} can be considered for every $i\in\mathbb Z$.

If $\I$ is invertible, then isomorphisms analogous to those
in \eqref{propcel2} follow, again by \cite{CEG}, Prop.~3.1, p.~13. Here, we will
display them in a format valid for every $i\in\mathbb Z$:
\begin{eqnarray}\label{propcel4}
\frac{\mathcal{I}^{(i)}|_{Y}}{\text{torsion}}\cong\I|_Y\ox\O_{Y}(q_-P)
\quad\text{and}\quad
\frac{\mathcal{I}^{(i)}|_{Z}}{\text{torsion}}\cong\I|_Z\ox\O_{Z}(q_+P),
\end{eqnarray}
where $q_-$ (resp.~$q_+$) is the quotient of the Euclidean division of
$-i$
(resp.~$i$) by $\delta$. These isomorphisms describe the $\I^{(i)}|_X$
completely: For $i\in\delta\mathbb Z$, the sheaf $\I^{(i)}|_X$ is the
invertible sheaf on $X$ whose restrictions to $Y$ and $Z$ are given by
\eqref{propcel4}, whereas for $i\not\in\delta\mathbb Z$, the sheaf
$\I^{(i)}|_X$ is the direct sum of the sheaves in \eqref{propcel4}.

When we put together the bottom exact sequences of
Diagrams \eqref{propcel} and
\eqref{propcel3}, we obtain the following Diagram:
\begin{equation}\label{propcel5}
\xymatrix{
0\ar[r]& \frac{\mathcal{I}^{(i+1)}|_{Y}}{\text{torsion}}\ar@{=}[d]\ar[r]   &
\mathcal{I}^{(i)}|_{X}\ar[d]<1ex>^{\varphi^i}\ar[r]  &
\frac{\mathcal{I}^{(i)}|_{Z}}{\text{torsion}}\ar@{=}[d]\ar[r] &0\\
0 & \frac{\mathcal{I}^{(i+1)}|_{Y}}{\text{torsion}}\ar[l]   &
\mathcal{I}^{(i+1)}|_{X}\ar[u]<1ex>^{\varphi_i}\ar[l]  &
\frac{\mathcal{I}^{(i)}|_{Z}}{\text{torsion}}\ar[l] &0\ar[l]}.
\end{equation}
The two maps, $\phi_i\:\I^{(i+1)}\to\I^{(i)}$ and
$\phi^i\:\I^{(i)}\to\I^{(i+1)}$, both natural inclusions, the
second involving a multiplication by $t$, restrict to isomorphisms on
$X_\eta$, but restrict to maps on $X$ whose compositions both ways are
zero; they are the maps $\varphi_i$ and $\varphi^i$ in Diagram
\eqref{propcel5}.

\begin{defn}\label{twseq}
A
\emph{twist $\delta$-sequence} associated to an invertible sheaf $I$
on $X$ is a collection of
sheaves $I^{(i)}$ and maps $\varphi^i,\varphi_i$ indexed by $i\in\mathbb Z$
such that:
\begin{description}
\item (a) The sheaves $I^{(i)}$ are invertible with restrictions
$I|_Y\ox\O_Y(q_-P)$ and $I|_Z\ox\O_Z(q_+P)$ if $\delta|i$, and
$$
I^{(i)}=\Big(I|_Y\ox\O_Y(q_-P)\Big)\bigoplus\Big(I|_Z\ox\O_Z(q_+P)\Big)
$$
otherwise;
here, $q_-$ (resp.~$q_+$) is the quotient of the Euclidean division of
$-i$
(resp.~$i$) by $\delta$.
\item (b) The Diagram
\begin{equation}\label{propcel6}
\xymatrix{
0\ar[r] & I|_Y\ox\O_Y(q'_-P)\ar@{=}[d]\ar[r]^{\quad\quad\quad\iota_{i,Y}} &
I^{(i)}\ar[d]<1ex>^{\varphi^i}\ar[r]^{\rho_{i,Z}\quad\quad\quad} &
I|_Z\ox\O_Z(q_+P)\ar@{=}[d]\ar[r] & 0\\
0 & I|_Y\ox\O_Y(q'_-P)\ar[l] &
I^{(i+1)}\ar[u]<1ex>^{\varphi_i}\ar[l]^{\quad\quad\,\,\rho_{i+1,Y}} &
I|_Z\ox\O_Z(q_+P)\ar[l]^{\iota_{i+1,Z}\quad\quad} & 0\ar[l],}
\end{equation}
commutes, where $q_+$ is as before, $q'_-$ is the quotient of the Euclidean
division of $-(i+1)$ by $\delta$, the maps $\rho_{i+1,Y}$ and
$\rho_{i,Z}$ are
the natural
surjections, and $\iota_{i,Y}$ and $\iota_{i+1,Z}$ are the natural
inclusions.
\end{description}
\end{defn}

Observe that $q'_-=q_-$ if $\delta$ does not divide $i$,
whereas $q'_-=q_--1$ otherwise. Similarly, the quotient of the Euclidean division
of $(i+1)$ by $\delta$ is $q_+$ if $\delta$ does not divide $i+1$, and $q_++1$
otherwise. Hence it follows from the description of the sheaves $I^{(i)}$ in
item (a) what the natural inclusions $\iota_{i,Y}$ and
$\iota_{i+1,Z}$ are. Also, notice that any other twist
$\delta$-sequence associated to $I$ is essentially the same, modulo
obvious isomorphisms.

If $\I$ is invertible, it follows from the isomorphisms
\eqref{propcel4} that the sheaves $\I^{(i)}|_X$ and
the maps $\varphi^i$ and $\varphi_i$ of Diagrams \eqref{propcel5} form
a twist $\delta$-sequence associated to $\I|_X$.

\section{Limit linear series}\label{lls3}

As in Section \ref{Tw}, we let $X$ denote a
curve defined over an algebraically closed field
$k$ with exactly two irreducible components, $Y$
and $Z$, that are smooth and intersect transversally at a single point
$P$. 
We let $\pi\colon\mathcal X\to B$ be a smoothing of
$X$, denote by $\eta$ and $o$ the
generic and special points of $B$, and by $X_\eta$ and $X_o$ the
respective fibers of $\pi$. We identify $X_o$ with $X$.

Let $\I$ be a torsion-free rank-1 sheaf on $\X/B$, and
let $\I^{(i)}$ denote the twists of $\I$ for $i\in\mathbb Z$, as
defined in Section \ref{Tw}. Let $I_\eta$ and $I^{(i)}_\eta$ denote
their restrictions to the generic fiber $X_\eta$; they are all equal.
Let $V\subseteq\Gamma(X_\eta,I_\eta)$ be a
vector subspace. Let $r$ denote its projective dimension.
View $V$ as a subspace of
$\Gamma(X_\eta,I_\eta^{(i)})$ for each $i\in\mathbb Z$, and denote by
$\V^{(i)}$ the subsheaf of $\pi_*\I^{(i)}$ consisting of the
sections that restrict to sections in $V$ on the generic fiber.  The
$\V^{(i)}$ are free of rank $r+1$. Let $\phi_i\:\I^{(i+1)}\to\I^{(i)}$ and
$\phi^i\:\I^{(i)}\to\I^{(i+1)}$ be the natural inclusions. Then
\begin{equation}\label{compat}
\begin{aligned}
\pi_*\phi_i(\V^{(i+1)})\subseteq\V^{(i)}&\text{ and }
(\pi_*\phi_{i})^{-1}(\V^{(i)})=\V^{(i+1)}\\
\pi_*\phi^i(\V^{(i)})\subseteq\V^{(i+1)}&\text{ and }
(\pi_*\phi^{i})^{-1}(\V^{(i+1)})=\V^{(i)}
\end{aligned}
\end{equation}

For each $i\in\mathbb Z$, let $V^{(i)}\subseteq\Gamma(X,\I^{(i)}|_X)$ 
be the subspace generated by
the restriction to $X$ of
$\Gamma(B,\V^{(i)})\subseteq\Gamma(\X,\I^{(i)})$. Let $\varphi^i$ and
$\varphi_i$ denote the restrictions of $\phi^i$ and $\phi_i$ to
$X$. Then it follows from \eqref{compat} that
$\varphi^i(V^{(i)})\subseteq V^{(i+1)}$ and
$\varphi_i(V^{(i+1)})\subseteq V^{(i)}$ for every $i$. Moreover,
since $\text{Ker}(\phi^i)=\text{Im}(\phi_i)$ and
$\text{Ker}(\phi_i)=\text{Im}(\phi^i)$, it follows that
$$
\varphi^i(V^{(i)})=\text{Ker}(\varphi_i|_{V^{(i+1)}})\text{ and }
\varphi_i(V^{(i+1)})=\text{Ker}(\varphi^i|_{V^{(i)}}).
$$
If $\I$ is invertible, then the data $(\I|_X;V^{(i)},\,i\in\mathbb Z)$
is an exact level-$\delta$ limit linear series, as defined below.

\begin{defn}\label{def:1} A
{\it level-$\delta$ limit linear series} of $X$ is the data
$\mathfrak g=(I;V^{(i)},\,i\in\mathbb Z)$ of an
invertible sheaf $I$ on $X$ and vector subspaces
$V^{(i)}\subseteq\Gamma(X,I^{(i)})$ for
$i\in\mathbb Z$ of equal dimension such that
$$
\varphi^i(V^{(i)})\subseteq
V^{(i+1)}\text{ and }
\varphi_i(V^{(i+1)})\subseteq
V^{(i)}
$$
for every $i$, where
the sheaves $I^{(i)}$ and the maps $\varphi^i,\varphi_i$ form
the twist $\delta$-sequence associated to $I$. We say that $\mathfrak
g$ has degree $d$ if $I$ has degree $d$, and that $\mathfrak g$ has
rank $r$ if the $V^{(i)}$ have projective dimension $r$. We say that
$\mathfrak g$ is \emph{exact} if moreover
$$
\varphi^i(V^{(i)})=\text{Ker}(\varphi_i|_{V^{(i+1)}})\text{ and }
\varphi_i(V^{(i+1)})=\text{Ker}(\varphi^i|_{V^{(i)}}).
$$
\end{defn}

Though it may seem that the data giving $\mathfrak g$ depend on
infinitely many parameters, this is not true. In fact, there are
integers $i_0$ and $i_d$ such that $I^{(i_0)}$ and $I^{(i_d)}$ are
invertible sheaves, the first with degree $0$ on $Z$, the second with
degree $0$ on $Y$. (Notice that $i_d=i_0+d\delta$.) Then, each
$s\in\Gamma(X,I^{(i)})$ for each $i\leq i_0$ (resp.~$i\geq i_d$) that
vanishes on $Y$ (resp.~$Z$) vanishes on the whole
$X$. This means that $\varphi_{i}|_{V^{(i+1)}}$ (resp.~$\varphi^i|_{V^{(i)}}$) is injective, and thus
$\varphi_{i}(V^{(i+1)})=V^{(i)}$ for every $i<i_0$
(resp.~$\varphi^i(V^{(i)})=V^{(i+1)}$ for every $i\geq i_d$). To
summarize, the $V^{(i)}$ for $i<i_0$ and $i>i_d$ are determined by the
$V^{(i)}$ for $i_0\leq i\leq i_d$. Furthermore, identifying the
level-$\delta$ limit linear series up to shifting, we may assume that
$i_0=0$ and thus $i_d=d\delta$.

If $I$ has degree 0 on $Z$ and $\delta=1$, then the truncation
$\ol{\mathfrak g}=(I,V^{(0)},\dots,V^{(d)})$ of a level-$\delta$
limit linear series $\mathfrak g$ is precisely a limit linear series in the sense
given by Osserman.

For each integer $d$, let ${\rm Pic}^{d}(X)$ denote the degree-$d$ Picard scheme of
$X$, parametrizing invertible sheaves of (total) degree $d$ on $X$.
It decomposes as the disjoint union of the subschemes
${\rm Pic}^{d-i,i}(X)$,
parameterizing invertible sheaves of bidegree
$(d-i,i)$, that is, degrees $d-i$ on $Y$ and $i$ on $Z$. Since $X$ is
of compact type, the restrictions give rise to isomorphisms
${\rm Pic}^{d-i,i}(X)\cong{\rm Pic}^{d-i}(Y)\times{\rm Pic}^i(Z)$.

We may now construct a moduli space for level-$\delta$ limit linear
series inside a product of relative Grassmannians over
$J:={\rm Pic}^{(d,0)}(X)$.

\begin{prop}\label{p1} There exists a projective scheme
$G_{d,\delta}^r(X)$ parameterizing
level-$\delta$ limit linear series of degree $d$ and rank $r$ on $X$.
\end{prop}

\begin{proof} The construction of $G^r_{d,\delta}(X)$ follows the same
  argument given to \cite{Oss1}, Thm.~5.3, p.~1178. To summarize it,
  let $\P$ be the Poincar\'e sheaf on $X\times J$, trivialized at
  $P$. Let $D$ be an ample enough effective Cartier divisor of $X$.
For each $i=0,\dots,d\delta$, let
$$
\P^{(i)}:=p_1^*\O^{(i)}_X\ox\P,\quad
\W^{(i)}:=p_{2*}(p_1^*\O_X(D)\ox\P^{(i)}),\text{ and }
\W_D^{(i)}:=p_{2*}(p_1^*(\O_X(D)|_D)\ox\P^{(i)}),
$$
where $p_1$ and $p_2$ are
the projections of $X\times J$ onto the indicated factors. The
$\O^{(i)}_X$ are the sheaves in the twist $\delta$-sequence associated
to $\O_X$. Let $\varphi_i,\varphi^i$ be the associated maps. They
induce maps $h_i,h^i$ between the $\W^{(i)}$. Also, restriction to $D$
induces maps $w_i\:\W^{(i)}\to\W_D^{(i)}$. The $\W_D^{(i)}$ are
locally free. If $D$ is ample
enough, so are the $\W^{(i)}$. Let
$G_i:=\text{Grass}_J(r+1,\W^{(i)})$ and set
$G:=G_0\times\cdots\times G_{d\delta}$.

Then $G_{d,\delta}^r(X)$ is a determinantal subscheme of $G$. Indeed,
let $\V^{(i)}\subseteq\W^{(i)}\ox\O_G$ be the pullback of the
universal subbundle of $\W^{(i)}$ from $G_i$ to $G$, for each
$i=0,\dots,d\delta$. The first condition we impose is that the compositions
$$
\begin{CD}
\V^{(i+1)}\hookrightarrow\W^{(i+1)}\ox\O_G @>h_i\ox\O_G>>
\W^{(i)}\ox\O_G \to
\frac{\W^{(i)}\ox\O_G}{\V^{(i)}}\\
\V^{(i)}\hookrightarrow\W^{(i)}\ox\O_G @>h^i\ox\O_G>> \W^{(i+1)}\ox\O_G \to
\frac{\W^{(i+1)}\ox\O_G}{\V^{(i+1)}}
\end{CD}
$$
be zero. The second condition is that the compositions
$$
\begin{CD}
\V^{(i)}\hookrightarrow\W^{(i)}\ox\O_G @>w_i\ox\O_G>>
\W^{(i)}_D\ox\O_G
\end{CD}
$$
be zero. These sets of conditions define $G_{d,\delta}^r(X)$.
\end{proof}

The points of $G^r_{d,\delta}(X)$ correspond to level-$\delta$ limit
linear series of $X$ up to shifting. More precisely,
$G^r_{d,\delta}(X)$
represents the functor that associates to each scheme $T$ an
invertible sheaf $\I$ on $X\times T$ of relative degree $d$ over $T$
and a collection of locally free
subsheaves $\V^{(i)}$ of rank $r+1$ of $p_{2*}(\I\ox p_1^*\O_X^{(i)})$
such that the pairs $(\I\ox p_1^*\O_X^{(i)},\V^{(i)})$ are families of
linear series on $X\times T/T$, and such that
$$
p_{2*}(\I\ox p_1^*(\varphi_i))(\V^{(i+1)})\subseteq\V^{(i)}
\text{ and }
p_{2*}(\I\ox p_1^*(\varphi^i))(\V^{(i)})\subseteq\V^{(i+1)},
$$
where the sheaves $\O_X^{(i)}$ and the maps $\varphi_i,\varphi^i$
form a twist $\delta$-sequence associated to $\O_X$, and $p_1$
and
$p_2$ are the projections of $X\times T$ onto the indicated
factors. Furthermore, either we identify the above objects up
to shifting, or we assume that $\I|_{Y\times T}$ has relative degree
$d$ over $T$. We may truncate the collection $(\V^{(i)},\,i\in\mathbb
Z)$ in the appropriate range or not, it does not matter. Finally, we have used the following
definition.

\begin{defn} Let $f\:M\to N$ be a proper flat map of schemes. A
\emph{family of linear series} on $M/N$ is the data of an invertible
sheaf $\L$ on $M$ and a locally free subsheaf $\V\subseteq f_*\L$
such that, for each Cartesian diagram
$$
\begin{CD}
M'@>u'>> M\\
@Vf'VV @VfVV\\
N'@>u>> N,
\end{CD}
$$
the composition of natural maps $u^*\V\to u^*f_*\L\to f'_*(u')^*\L$ is
injective.
\end{defn}

Associated to a level-$\delta$ limit linear series, we have a
collection of vector spaces $(V^{(i)},\,i\in\mathbb Z)$ and maps
$h_i:=\Gamma(\varphi_ i)$ and $h^i:=\Gamma(\varphi^i)$ between
them. These data constitute a \emph{linked sequence of vector spaces},
according to the definition below, following Osserman.

\begin{defn} A \emph{linked sequence of vector spaces} is the data of
  a collection of vector spaces $(V^{(i)},\,i\in\mathbb Z)$ of the
  same dimension and maps
$h_i\:V^{(i+1)}\to V^{(i)}$ and $h^i\:V^{(i)}\to V^{(i+1)}$ satisfying
the following conditions:
\begin{description}
\item (a) $h_ih^i=0$ and $h^ih_i=0$ for every $i$.
\item (b) $\text{Ker}(h^i)\cap\text{Ker}(h_{i-1})=0$ for every $i$.
\item (c) There are integers $i_0$ and $i_\infty$ such that $h_i$ is an
  isomorphism for every $i< i_0$ and $h^i$ is an isomorphism for every
  $i\geq i_\infty$.
\end{description}
The linked sequence is called \emph{exact} if the complex
$$
\begin{CD}
V^{(i)} @>h^i>> V^{(i+1)} @>h_i>> V^{(i)} @>h^i>> V^{(i+1)}
\end{CD}
$$
is exact for every $i$. The \emph{dimension} of the sequence is the
dimension of the $V^{(i)}$. Its \emph{lower bound} (resp.~\emph{upper bound}) is the
maximum $i_0$ (resp.~minimum $i_\infty$) for which $h_i$ (resp.~$h^i$) is
an
isomorphism for every $i<i_0$ (resp.~$i\geq i_\infty$).
\end{defn}

Notice that, if $h_i$ is an isomorphism, so is $h_{i-1}$.
Indeed, since $\Im(h_i)\subseteq\text{Ker}(h^i)$, if $h_i$ is an
isomorphism, then $h^i=0$; whence, since
$\text{Ker}(h^i)\cap\text{Ker}(h_{i-1})=0$, it follows that $h_{i-1}$
is injective, thus an isomorphism. Analogously, if $h^i$ is an
isomorphism, so is $h^{i+1}$. Thus the lower bound (resp.~upper bound)
is the maximum $i$ (resp.~minimum $i$) for which $h_{i-1}$ (resp. $h^i$)
is an isomorphism.

There are many ways to characterize exactness.

\begin{prop}\label{pqm}
Let $(V^{(i)},h^i,h_i\,|\,i\in\mathbb Z)$ be a sequence
  of linked vector spaces of dimension $n$. For each $i$, let
$$
p_i:=\dim\Ker(h_{i-1}),\quad q_i:=\dim\Ker(h^i)\text{ and
}
m_i:=n-p_i-q_i.
$$
Then
\begin{description}
\item {\rm (a)} $p_i+q_i+m_i=n$ for every $i$;
\item {\rm (b)} $p_i,q_i,m_i\geq 0$ for every $i$;
\item {\rm (c)} $(p_i,q_i,m_i)=(0,n,0)$ for every
$i<<0$ and
$(p_i,q_i,m_i)=(n,0,0)$ for every $i>>0$.
\item {\rm (d)} $p_i+m_i\leq p_{i+1}$ for every $i$;
\item {\rm (e)} $q_{i+1}+m_{i+1}\leq q_i$ for every $i$;
\item {\rm (f)} $\sum m_i\leq n$;
\item {\rm (g)} $\rk(h^i)+\rk(h_i)\leq n$ for every $i$.
\end{description}
Furthermore, equalities hold in {\rm (d)} for every $i$ if and only if they
hold in {\rm (e)}, if and only if they hold in {\rm (g)}, if and only
if equality holds in {\rm (f)}, if and only if the
sequence is exact.
\end{prop}

\begin{proof} Since $\Ker(h_{i-1})\cap\Ker(h^i)=0$, it follows that
  $p_i+q_i\leq n$, yielding the only nontrivial part of (a) and
  (b). Furthermore, (c) follows from the fact that
$h_i$ is an isomorphism for $i<<0$ and $h^i$ is an
isomorphism for $i>>0$.

In addition, for each $i$,
\begin{align*}
p_i+m_i=n-q_i=\dim\Im(h^i)&\leq\dim\Ker(h_i)=p_{i+1},\\
q_{i+1}+m_{i+1}=n-p_{i+1}=\dim\Im(h_i)&\leq\dim\Ker(h^i)=q_i.
\end{align*}
Also,
$$
\rk(h^i)+\rk(h_i)=n-q_i+n-p_{i+1}=p_i+m_i+n-p_{i+1}=n-q_i+q_{i+1}+m_{i+1}.
$$
Thus (d), (e) and (g) follow, as well as the equivalence between the
equalities in (d), (e) or (g) and exactness. Finally, since $p_i=0$ for $i<<0$ and
$p_i=n$ for $i>>0$, it follows from (d) that
$$
\sum m_i\leq\sum(p_{i+1}-p_i)= n,
$$
with equality if and only if equalities hold in (d) for every $i$.
\end{proof}

The above proposition suggests a definition.

\begin{defn} Let $f\:\z\to\z^3$. For each $i\in\z$, let
$(p_i,q_i,m_i):=f(i)$. We say that $f$ is
\emph{$n$-admissible} if Conditions (a)-(f) in
Proposition \ref{pqm} are verified. Furthermore,
we say that $f$ is \emph{exact} if all the
inequalities in (d)-(f) are equalities. If the $p_i,q_i,m_i$ are as in
Proposition \ref{pqm}, we say that $f$ is the
\emph{numerical function} of the sequence of linked
vector spaces.
\end{defn}

Notice that, if $f$ is exact, then the $p_i$ and $q_i$
are determined from the $m_i$. Also,
by Proposition~\ref{pqm},
the numerical function of a sequence of linked vector spaces is exact
if and only if the sequence is exact.

Proposition \ref{pqm} suggests a stratification of
$G^r_{d,\delta}(X)$. Indeed, to each $\mathfrak g\in
G^r_{d,\delta}(X)$ assign the
numerical function $f_{\mathfrak g}$ of the
sequence of linked vector
spaces arising from $\g$. And, for each
$(r+1)$-admissible $f\:\z\to\z^3$, let
$$
G_{d,\delta}^r(X;f):=\left\{\mathfrak{g}\in
  G_{d,\delta}^r(X)\,|\, f_{\mathfrak g}=f\right\}.
$$
Since rank is semicontinuous, $G_{d,\delta}^r(X;f)$ is a locally closed subset of
$G_{d,\delta}^r(X)$.

It follows from Proposition \ref{pqm} that
$\g$ is exact if and only if $f_\g$ is exact. Thus the
subset
$$
G^{r,*}_{d,\delta}(X)\subseteq
G^{r}_{d,\delta}(X),
$$
parameterizing exact level-$\delta$ limit linear
series, decomposes as
$$
G_{d,\delta}^{r,*}(X)=\bigcup_{f\text{ exact}}
G_{d,\delta}^r(X;f).
$$
By semicontinuity, $G_{d,\delta}^{r,*}(X)$ is open
in $G_{d,\delta}^r(X)$. Furthermore,
the $G_{d,\delta}^r(X;f)$, for $f$ exact, are both open and
closed in $G_{d,\delta}^{r,*}(X)$.


\section{The forgetful maps}

As in Section \ref{Tw}, let $X$ denote a
curve defined over an algebraically closed field
$k$ with exactly two irreducible components, $Y$
and $Z$, that are smooth and intersect transversally at a single point
$P$. 
Let $d$ and $r$ be integers.

There are natural ``forgetful'' morphisms
$$
\rho_{\delta',\delta}\:G_{d,\delta'}^r(X)\longrightarrow G_{d,\delta}^r(X)
$$
for $\delta$ and $\delta'$ such that $\delta|\delta'$. Indeed, if $I$
is an invertible sheaf on $X$ and
  $(I^{(i)},\varphi^i,\varphi_i\,|\,i\in\mathbb Z)$ is the
twist $\delta'$-sequence associated to $I$, then
$(I^{(ci)},\varphi_{ci}^{c(i+1)},\varphi_{c(i+1)}^{ci}\,|\,i\in\mathbb Z)$ is the twist
$\delta$-sequence associated to $I$, where $c:=\delta'/\delta$ and
\begin{equation}\label{fij}
\varphi_i^j:=\begin{cases}
\varphi^{j-1}\cdots\varphi^{i+1}\varphi^{i}&\text{if $j>i$,}\\
\varphi_j\varphi_{j+1}\cdots\varphi_{i-1}&\text{if $j<i$,}\\
\text{id}&\text{if $j=i$.}
\end{cases}
\end{equation}
Thus $\rho_{\delta',\delta}$ is well-defined by taking a
level-$\delta'$ limit linear series $\mathfrak
g'=(I;V^{(i)},i\in\mathbb Z)$ to the level-$\delta$ limit linear series
$\mathfrak g=(I;V^{(ci)},i\in\mathbb Z)$.

Clearly, the
forgetful maps satisfy
$$
\rho_{\delta',\delta}\rho_{\delta'',\delta'}=\rho_{\delta'',\delta}
\text{ if }\delta|\delta'|\delta''.
$$
Furthermore,
\begin{equation}\label{fcf}
\rho_{\delta',\delta}(G_{d,\delta'}^r(X;f'))\subseteq
G_{d,\delta}^r(X,f),\text{ where }f(i)=f'(ci)\text{ for every }i.
\end{equation}
Indeed, for any $\mathfrak g'=(I;V^{(i)},i\in\mathbb Z)$,
the kernel of $\varphi_{i-1}$ is the same as that
of the composition $\varphi_i^j$ for
every $j<i$. Likewise, the kernel of $\varphi^{i}$ is the same as that
of the composition $\varphi_i^j$ for
every $j>i$. Thus:
\begin{equation}\label{fcf2}
\text{If $\mathfrak g=\rho_{\delta',\delta}(\mathfrak
g')$, then $f_{\mathfrak g}(i)=f_{\mathfrak g'}(ci)$ for every $i$.}
\end{equation}

Given $f,f'\colon\mathbb Z\to\mathbb Z^3$ and
$c\in\mathbb Z$,
we write $f=f'c$ when
$f(i)=f'(ci)$ for every $i$.

\begin{thm}\label{p2} Let
$f\:\mathbb Z\to\mathbb Z^3$ be $(r+1)$-admissible.
If $\delta'=c\delta$ then
$$
\rho_{\delta',\delta}^{-1}(G_{d,\delta}^{r}(X;f))=
\bigcup_{f=f'c}G_{d,\delta'}^{r}(X;f').
$$
Furthermore, if $f=f'c$, and $f'$ is $(r+1)$-admissible,
then the restriction
$$
\rho_{\delta',\delta}\colon G_{d,\delta'}^{r}(X;f')\longrightarrow
G_{d,\delta}^{r}(X;f)
$$
is surjective and $G_{d,\delta}^r(X;f)$-isomorphic to a nonempty open
subscheme of a product of relative Grassmannians over
a product of relative partial flag varieties over
$G_{d,\delta}^{r}(X;f)$, and has relative dimension
\begin{align*}
\sum_{i\in\z}\Big(&\sum_{j=1}^{c-1}
(q_{ci+j-1}-q_{ci+j})(p_{ci+c}-p_{ci+j}-m_{ci+j})+\\
&\sum_{j=1}^{c-1}
(p_{ci+j+1}-p_{ci+j})(q_{ci}-q_{ci+j}-m_{ci+j})+\\
&\sum_{j=1}^{c-1}
m_{ci+j}(p_{ci+j+1}-p_{ci+j-1}-m_{ci+j-1})\Big),
\end{align*}
where $(p_i,q_i,m_i):=f'(i)$ for each $i\in\z$. In
particular, if $f'$ is exact, the relative dimension is
$$
\sum_{i\in\z}(m_{ci+1}+\cdots+m_{ci+c-1})^2
$$
\end{thm}

\begin{proof} The first statement follows from
  \eqref{fcf2}. We prove now the second statement.

Let $(p_i,q_i,m_i):=f'(i)$ for each $i\in\z$.
Let $I$ be an invertible sheaf on $X$ and
  $(I^{(i)},\varphi^i,\varphi_i\,|\,i\in\mathbb Z)$ the
twist $\delta'$-sequence associated to $I$. Then
$(I^{(ci)},\varphi^{c(i+1)}_{ci},\varphi_{c(i+1)}^{ci}\,|\,i\in\z)$ is the twist
$\delta$-sequence associated to $I$.
For each $i\in\z$,
let $I^{(i)}_Y$ (resp.~$I^{(i)}_Z$) denote the
restriction of $I^{(i)}$ to $Y$ (resp.~$Z$) modulo torsion. Then
$$
I^{(ci+j)}=I^{(c(i+1))}_Y\oplus I^{(ci)}_Z
$$
for each $j=1,\dots,c-1$.

For each $i\in\mathbb Z$, and each subspace
$V\subseteq \Gamma(X,I^{(i)})$ of dimension $r+1$, let
$h_-^V:=\varphi_{i-1}|_{V}$ and
$h_+^V:=\varphi^i|_{V}$, set
\begin{align*}
p(V):=&\dim\Ker(h_-^V),\\
q(V):=&\dim\Ker(h_+^V),\\
m(V):=&r+1-p(V)-q(V),
\end{align*}
and let $V_-$
(resp.~$V_+$) denote the image of $V$ in
$\Gamma(Y,I^{(i)}_Y)$ (resp.~$\Gamma(Z,I^{(i)}_Z)$).

For each $i\in\z$, let
$V^{(ci)}\subseteq\Gamma(X,I^{(ci)})$ be a subspace
such that
$\mathfrak{g}:=(I; V^{(ci)},\,i\in\mathbb Z)$ is a
level-$\delta$ limit linear series of rank $r$ of
$X$ with $f_\g=f$. Since
the kernel of $\varphi^i$ is the same as that of
$\varphi^j_i$ for every $j>i$, and likewise for
$\varphi_{i-1}$, it follows that
$$
f(i)=(p(V^{(ci)}),q(V^{(ci)}),m(V^{(ci)}))\text{ for every }i\in\z.
$$
Also,
for each $i\in\mathbb Z$, let $W^{(ci)}_+$
(resp.~$W^{(c(i+1))}_-$) be the subspace of
$\Gamma(Z,I^{(ci)}_Z)$ (resp.~$\Gamma(Y,I^{(c(i+1))}_Y)$)
such that
$$
\varphi^{ci+c-1}(0\oplus W^{(ci)}_+)=\Ker(h_-^{V^{(c(i+1))}})
\quad\text{(resp.~ }
\varphi_{ci}(W^{(c(i+1))}_-\oplus 0)=\Ker(h_+^{V^{(ci)}})\text{)}.
$$
Then $V^{(ci)}_+\subseteq W^{(ci)}_+$
and $V^{(c(i+1))}_-\subseteq W^{(c(i+1))}_-$.

If $\mathfrak{g'}:=(I;
V^{(i)},\,i\in\mathbb Z)$ is a level-$\delta'$ limit linear series of
rank $r$ such that $f_{\g'}=f'$ and
$\mathfrak g=\rho_{\delta',\delta}(\mathfrak g')$, then,
for each $i\in\z$,
\begin{align}
&V^{(ci)}_+\subseteq V^{(ci+1)}_+\subseteq\cdots\subseteq
V^{(c(i+1)-1)}_+\subseteq W^{(ci)}_+,\\
&V^{(c(i+1))}_-\subseteq V^{(ci+c-1)}_-\subseteq\cdots\subseteq
V^{(ci+1)}_-\subseteq W^{(c(i+1))}_-.
\end{align}
Notice that
$$
\dim V_+^{(ci+j)}=r+1-q_{ci+j}\text{ and }
\dim V_-^{(ci+j)}=r+1-p_{ci+j}
$$
for $j=1,\dots,c-1$. Also,
$$
V^{(ci+j+1)}_-\oplus V^{(ci+j-1)}_+\subseteq
V^{(ci+j)}\subseteq V^{(ci+j)}_-\oplus V^{(ci+j)}_+
\subseteq\Gamma(Y,I_Y^{(c(i+1))})\oplus
\Gamma(Z,I_Z^{(ci)}).
$$
Furthermore, the
projections $V^{(ci+j)}\to V^{(ci+j)}_-$ and
$V^{(ci+j)}\to V^{(ci+j)}_+$ are surjective.

Conversely, for each $i\in\z$, let
$A_{i,1},\dots,A_{i,c-1}$ and $B_{i,1},\dots,B_{i,c-1}$
be subspaces of $W^{(ci)}_+$ and
$W^{(ci+c)}_-$, respectively, with
$$
\dim A_{i,j}=r+1-q_{ci+j}\text{ and }
\dim B_{i,j}=r+1-p_{ci+j}\text{ for }j=1,\dots,c-1,
$$
such that
\begin{align*}
&V^{(ci)}_+\subseteq A_{i,1}\subseteq\cdots\subseteq
A_{i,c-1}\subseteq W^{(ci)}_+,\\
&V^{(c(i+1))}_-\subseteq B_{i,c-1}\subseteq\cdots\subseteq
B_{i,1}\subseteq W^{(c(i+1))}_-.
\end{align*}
This is possible since
\begin{align*}
\dim V^{(ci)}_+&=r+1-q_{ci}\leq r+1-q_{ci+j}\leq
r+1-q_{ci+j+1}\leq p_{ci+c}=\dim W^{(ci)}_+\\
\dim V^{(c(i+1))}_-&=r+1-p_{ci+c}\leq r+1-p_{ci+c-j}\leq
r+1-p_{ci+c-j-1}\leq q_{ci}=\dim W^{(ci+c)}_-
\end{align*}
for $j=1,\dots,c-2$. For each $j=1,\dots,c-1$, let
$V^{(ci+j)}\subseteq B_{i,j}\oplus A_{i,j}$ be any subspace
of dimension $r+1$ such that
$V^{(ci+j)}\supseteq B_{i,j+1}\oplus A_{i,j-1}$ and
such that the
projections $V^{(ci+j)}\to A_{i,j}$ and $V^{(ci+j)}\to B_{i,j}$
are surjective. This is possible, since
$$
\dim(B_{i,j+1}\oplus A_{i,j-1})=2(r+1)-p_{ci+j+1}-
q_{ci+j-1}\leq r+1\leq 2(r+1)-p_{ci+j}-q_{ci+j}=
\dim(B_{i,j}\oplus A_{i,j}).
$$
Then $\mathfrak{g'}:=(I;
V^{(i)},\,i\in\mathbb Z)$ is a level-$\delta'$ limit linear
series of rank $r$ such that $f_{\g'}=f'$ and
$\mathfrak g=\rho_{\delta',\delta}(\mathfrak g')$.

If follows that
$\rho_{\delta',\delta}^{-1}(\mathfrak g)\cap
G_{d,\delta'}^r(X; f')$ is parameterized by a nonempty
open subset of a product of Grassmannians over a
product of flag varieties. The flag varieties are
\begin{align*}
F_+^{(i)}:=
&\{\ol A_{i,1}\subseteq\cdots\subseteq
\ol A_{i,c-1}\subseteq
\frac{W^{(ci)}_+}{V^{(ci)}_+}\,|\,
\dim\ol A_{i,j}=q_{ci}-q_{ci+j}\text{ for }j=1,\dots,c-1\},\\
F_-^{(i)}:=&\{\ol B_{i,c-1}\subseteq\cdots\subseteq
\ol B_{c,1}\subseteq
\frac{W^{(c(i+1))}_-}{V^{(c(i+1))}_-}\,|\,
\dim\ol B_{i,c-j}=
p_{c(i+1)}-p_{c(i+1)-j}\text{ for }j=1,\dots,c-1\}.
\end{align*}
And the Grassmannians over points
$(\ol A_{i,j})\in F^{(i)}_+$ and $(\ol B_{i,c-j})\in F^{(i)}_-$
are
$$
\text{Grass}
\Big(p_{ci+j+1}+q_{ci+j-1}-(r+1),
\frac{\ol B_{i,j}}{\ol B_{i,j+1}}\oplus
\frac{\ol A_{i,j}}{\ol A_{i,j-1}}
\Big).
$$
The relative dimension can
thus be computed from the formulas for the dimensions of relative flag varieties.
\end{proof}

\begin{lem}\label{lfff}
Let $f\:\z\to\z^3$ be an $n$-admissible function. Let $c>1$ be an
integer. Then there
is an exact $n$-admissible function
$f'$ such that $f=f'c$. If $f$ is exact, then there is a
unique $n$-admissible function
$f'$ such that $f=f'c$.
\end{lem}

\begin{proof}  Let $(p_i,q_i,m_i):=f(i)$ for
each $i\in\z$. Let $\ell\in\{1,\dots,c-1\}$. Set $f'(ci):=f(i)$
for each $i$ and
\begin{equation}\label{fff}
f'(ci+j):=\begin{cases}
(p_i+m_i,q_i,0)&\text{ if $j<\ell$},\\
(p_{i}+m_{i},q_{i+1}+m_{i+1},p_{i+1}-p_{i}-m_{i})&\text{ if
  $j=\ell$},\\
(p_{i+1},q_{i+1}+m_{i+1},0)&\text{ if $j>\ell$}
\end{cases}
\end{equation}
for each $i$ and $j=1,\dots,c-1$; then
$f'$ is $n$-admissible, is exact and $f=f'c$.

On the other hand, if $f$ is exact and $f'$
is an $n$-admissible function such that $f=f'c$, then, since
$\sum m_i=n$, the maximum possible, it follows that
$f'(ci+j)\in\z^2\times\{0\}$ for $j=1,\dots,c-1$, and
that $f'$ is exact. Now, since $f'$ is exact, the
knowledge of the $m_i$ determines $f'$ uniquely.
\end{proof}

\begin{prop}\label{Cor1} If $\delta|\delta'$ then
$$
\rho_{\delta',\delta}^{-1}(G_{d,\delta}^{r,*}(X))\subseteq
G_{d,\delta'}^{r,*}(X),
$$
and the restriction
$$
\rho_{\delta',\delta}\colon \rho_{\delta',\delta}^{-1}(G_{d,\delta}^{r,*}(X))
\longrightarrow
G_{d,\delta}^{r,*}(X)
$$
is an isomorphism. Furthermore, if $\delta'>\delta$, then
$$
\rho_{\delta',\delta}(G_{d,\delta'}^{r,*}(X))=
G_{d,\delta}^{r}(X).
$$
\end{prop}

\begin{proof} It follows from Lemma \ref{lfff} and
the first statement of
Theorem \ref{p2} that, if $f$ is exact, then
$\rho_{\delta',\delta}^{-1}(G_{d,\delta}^{r}(X; f))
=G_{d,\delta'}^{r}(X; f')$, where $f'$ is the exact
function such that $f=f'c$, where
$c:=\delta'/\delta$.  This is enough to conclude
that $\rho_{\delta',\delta}^{-1}(G_{d,\delta}^{r,*}(X))\subseteq
G_{d,\delta'}^{r,*}(X)$, since
the $G_{d,\delta}^{r}(X; f)$ for $f$ exact cover $G_{d,\delta}^{r,*}(X)$ and
the $G_{d,\delta'}^{r}(X; f')$ for $f'$ exact cover $G_{d,\delta'}^{r,*}(X)$.

To prove the remaining of the first statement, since the
$G_{d,\delta}^{r}(X; f)$ for $f$ exact decompose
$G_{d,\delta}^{r,*}(X)$ into open subsets, it is enough to show that
the restriction 
$$
\rho_{\delta',\delta}\colon G_{d,\delta'}^{r}(X;f')
\longrightarrow G_{d,\delta}^{r}(X;f)
$$
is an isomorphism for every
exact $f$, where $f'$ is the unique exact function such that
$f=f'c$. By Theorem \ref{p2}, it is enough to show that the relative
dimension of the latter map is zero in this case. This relative
dimension is computed in Theorem \ref{p2} and turns out to be zero
because $f$ is exact.

To prove the second statement, by Theorem \ref{p2},
it is enough to show
that for each $(r+1)$-admissible function
$f\:\z\to\z^3$ there is an exact $(r+1)$-admissible
function $f'$ such that $f=f'c$, where
$c:=\delta'/\delta$. But this is exactly what
Lemma \ref{lfff} claims.
\end{proof}

\section{Abel maps}

\begin{defn}
Let $\mathfrak V:=(V^{(i)},h^i,h_i\,|\,i\in\z)$ be a sequence of linked vector
  spaces of dimension $n$. Let $m\in\z$. The \emph{elementary truncation at $m$} of
  $\mathfrak V$ is the sequence of linked vector spaces
$\mathfrak W:=(W^{(i)},f^i,f_i\,|\,i\in\z)$ where
\begin{description}
\item {\rm (a)} $W^{(i)}=V^{(i)}$ for $i\leq m$ and $W^{(i)}=V^{(i+1)}$
    for $i> m$;
\item {\rm (b)} $f^i=h^i$ for $i<m$ $f^m=h^{m+1}h^m$ and $f^i=h^{i+1}$ for
  $i>m$;
\item {\rm (c)} $f_i=h_i$ for $i<m$ $f_m=h_mh_{m+1}$ and $f_i=h_{i+1}$ for
  $i>m$.
\end{description}
A \emph{truncation} is the sequence of linked vector spaces obtained
after a finite sequence of elementary truncations. On the other
hand, we say that a sequence of linked vector spaces is an
\emph{expansion} of another, if the latter is a truncation of the former.
\end{defn}

Let $\mathfrak V:=(V^{(i)},h^i,h_i\,|\,i\in\z)$ be a sequence of linked vector
  spaces of dimension $n$. Let $i_0$ be its lower bound and
  $i_{\infty}$ its upper bound. Let $i\leq i_0$ and $j\geq
  i_\infty$. Put $\IP^{1,1}(\mathfrak V):=\mathbb P(V^{(i)})\times\mathbb P(V^{(j)})$.
In principle, $\IP^{1,1}(\mathfrak V)$
is defined up to the choice of $i$ and $j$. However, since $h^\ell$ is
an isomorphism for $\ell\geq i_\infty$ and $h_\ell$ is an isomorphism
for $\ell\leq i_0-1$, the scheme is the same up to natural
isomorphism.

Let
$$
\mathbb P(\mathfrak V):=\bigcup_{\ell=i}^j\IP(\mathfrak
V_\ell)\subseteq \IP^{1,1}(\mathfrak V),
$$
where
$$
\IP(\mathfrak V_\ell):=
\ol{\{([h_\ell^i(v)],[h_\ell^j(v)])\in\mathbb P^{1,1}(\mathfrak V)\,|\,v\in
  V^{(\ell)}-(\Ker(h^{\ell})\cup\Ker(h_{\ell-1}))\}}.
$$
Here,
$$h^i_\ell:=\begin{cases}
h_ih_{i+1}\cdots h_{\ell-1}&\text{if $i<\ell$,}\\
h^{i-1}\cdots h^{\ell+1}h^\ell&\text{if $i>\ell$,}\\
\text{id}&\text{if $i=\ell$.}
\end{cases}
$$
Observe that $\Ker(h_\ell^i)=\Ker(h_{\ell-1})$ for $\ell> i$ and
$\Ker(h_\ell^j)=\Ker(h^\ell)$ for $\ell< j$. Thus $h^i_\ell(v)$ and
$h^j_\ell(v)$ are nonzero for $v\in
V^{(\ell)}-(\Ker(h^{\ell})\cup\Ker(h_{\ell-1}))$ and define points
in $\IP(V^{(i)})$ and $\IP(V^{(j)})$.

Notice that $\IP(\mathfrak V)$ comes with natural invertible sheaves
$\O_{\IP(\mathfrak V)}(i,j)$, obtained by restriction of the sheaves
$\O(i,j)$ on $\IP^{1,1}(\mathfrak V)$, which do not depend of the
particular $\IP^{1,1}(\mathfrak V)$ chosen. Also, if $\mathfrak W$ is a
truncation of
$\mathfrak V$, then
$\P(\mathfrak W)\subseteq\IP(\mathfrak V)$ naturally, with
$\O_{\IP(\mathfrak W)}(i,j)=\O_{\IP(\mathfrak V)}(i,j)|_{\IP(\mathfrak W)}$ for all $i,j$.

As in Section \ref{Tw}, let $X$ denote a
curve defined over an algebraically closed field
$k$ with exactly two irreducible components, $Y$
and $Z$, that are smooth and intersect transversally at a single point
$P$. 
Let $d$ and $r$ be integers. Let $J$ be the connected component
of the Picard scheme of $X$ parameterizing invertible sheaves of
degree $d$ on $Y$ and $0$ on $Z$. Recall that $T^{(i)}$ denotes the
$i$th symmetric product of a scheme $T$.

\begin{defn}\label{5.2} The \emph{degree-$d$ Abel map} $A_d\colon X^{(d)}\to J$
associates to each
Weil divisor $D$ of $X$ the invertible sheaf $\O_X(D)$ defined as that
having restrictions $\O_Y(D_1+d_2P)$ and $\O_Z(D_2-d_2P)$, where
$D=D_1+D_2$, with $D_1$ supported in $Y$ of degree $d_1$ and $D_2$
supported in $Z$ of degree $d_2$.
\end{defn}

By its very definition, $\O_X(D)$ does not depend on
the decomposition $D=D_1+D_2$ chosen; see \cite{E-Oss}, $\S3$, p.~82.

Let $I$ be an invertible sheaf on $X$ and
$(I^{(i)},\varphi^i,\varphi_i\,|\,i\in\z)$ an associated twist
$\delta$-sequence. For each $i,j$, let $\varphi_i^j$ be as in
\eqref{fij}. Let $i_0$ be the integer for which $I^{(i_0)}$ is
invertible of degree $d$ on $Y$ and $i_{\infty}$ that for which
$I^{(i_\infty)}$ is invertible of degree $d$ on $Z$. Then
$i_\infty=i_0+d\delta$. We may assume, without loss of generality, that $i_0=0$.

Let $\mathfrak g=(I;V^{(i)},\,i\in\mathbb Z)$ be a level-$\delta$
limit linear series. Let $\mathfrak V:=(V^{(i)},h^i,h_i\,|\,i\in\z)$
be the associated sequence of linked vector spaces. Then $i_0$ is at
most its lower bound and $i_\infty$ is at least its upper bound. In
particular, we may view $\mathbb P(\mathfrak V)$ inside 
$\mathbb P(V^{(0)})\times\mathbb P(V^{(d\delta)})$.

Let $h_i^j:=\varphi_i^j|_{V^{(i)}}$
for every $i$ and $j$. To each $i\in\mathbb Z$ and $s\in V^{(i)}$ such
that $0\leq i\leq d\delta$ and $s\not\in\Ker(h^i)\cup\Ker(h_{i-1})$,
we associate the
point $[D(s)]$ on $X^{(d)}$ given by
$$
D(s):=\div(h_{i}^\ell(s)|_Y)+\div(h_i^m(s)|_Z)-
\begin{cases}
P&\text{if $\ell<m$},\\
0&\text{if $\ell=m$},
\end{cases}
$$
where $\ell$ (resp.~$m$) is the maximum (resp.~minimum) integer not
greater (resp.~not smaller) than $i$ such that $I^{(\ell)}$
(resp.~$I^{(m)}$) is invertible.

Notice that $D(s)$ has degree $d$. Indeed,
$\div(h_{i}^\ell(s)|_Y)$ has degree $d-\ell/\delta$, while
$\div(h_i^m(s)|_Z)$ has degree $m/\delta$. The sum has
degree $d+(m-\ell)/\delta$, from which follows that $D(s)$ has degree
$d$. Also, $D(s)$ is effective. This is clear if $\ell=m$. On the
other hand, if $\ell<m$, then
$I^{(i)}=I^{(m)}|_Y\oplus I^{(\ell)}|_Z$, and thus both
$h_{i}^\ell(s)|_Y$ and $h_{i}^m(s)|_Z$ vanish at 
$P$. At any rate, it follows from Definition \ref{5.2} that $\O_X(D(s))$ restricts to
$$
I^{(\ell)}|_Y\ox\O_Y\big(\frac{\ell}{\delta} P\big)\text{ and
  }I^{(m)}|_Z\ox\O_Z\big(d-\frac{\ell}{\delta}P\big),
$$
which are the same restrictions as those of $I^{(0)}$. Thus
$[D(s)]\in A_d^{(-1)}([I^{(0)}])$.

(Another way of viewing $D(s)$ is by considering the dual
$(I^{(i)})^*\to\O_X$ of the homomorphism induced by $s$; the image is
the sheaf of ideals of a finite subscheme of $X$ whose 0-cycle is
$D(s)$.)

Let
$$
\IP(\mathfrak g):=\bigcup_{i=0}^{d\delta}\IP(\mathfrak g_i)\subseteq X^{(d)},\text{
  where }
\IP(\g_i):=\ol{\{[D(s)]\,|\, s\in
  V^{(i)}-(\Ker(h^i)\cup\Ker(h_{i-1}))\}}.
$$
Then $\IP(\mathfrak g)\subseteq A_d^{(-1)}([I^{(0)}])$.

Notice that
\begin{equation}\label{d2d}
\div(\varphi_{i}^{0}(s)|_Y)+\div(\varphi_i^{d\delta}(s)|_Z)=D(s)+dP=\tau_{dP}([D(s)])
\end{equation}
for each $i$ and $s$, where $\tau_{dP}\:X^{(d)}\to X^{(2d)}$ is the
embedding taking $[D]$ to $[D+dP]$. Furthermore, there are natural inclusions
$\iota_0\:\IP(V^{(0)})\hookrightarrow Y^{(d)}$ and
$\iota_{d\delta}\:\IP(V^{(d\delta)})\hookrightarrow Z^{(d)}$, the first taking $[s]$ to
$\div(s|_Y)$, the second taking $[s]$ to $\div(s|_Z)$. Composing with
the sum embedding $\sigma\:Y^{(d)}\times Z^{(d)}\to X^{(2d)}$, it
follows from \eqref{d2d} that
$$
\sigma(\iota_0\times\iota_{d\delta})(\IP(\mathfrak
V))=\tau_{dP}(\IP(\mathfrak g)),
$$
Hence, we may naturally identify $\IP(\mathfrak V)$ with
$\IP(\g)$ inside $X^{(2d)}$. Furthermore, if we let $Q_Y\in Y-P$ and
$Q_Z\in Z-P$ be any two points, and set $H^{\ell}_{i,j}:=iH^{\ell}_Y+jH^{\ell}_Z$, where
$$
H^\ell_Y:=\{[D]\in X^{(\ell)}\,|\,D\geq Q_Y\}\text{ and }
H^\ell_Z:=\{[D]\in X^{(\ell)}\,|\,D\geq Q_Z\}
$$
for all $i,j,\ell$, then $\O_{\IP(\mathfrak V)}(i,j)$ and
$\O_{X^{(d)}}(H^{d}_{i,j})|_{\IP(\g)}$ coincide, both being restrictions of
$\O_{X^{(2d)}}(H^{2d}_{i,j})$.

\begin{lem}\label{trunc} Each sequence of linked vector spaces is a truncation of
  an exact sequence.
\end{lem}

\begin{proof} Let $\mathfrak V=(V^{(i)},h^i,h_i\,|\,i\in\z)$ be a
  sequence of linked vector spaces of dimension $n$. Recall the notation:
$$
p_i:=\dim\Ker(h_{i-1}),\quad q_i:=\dim\Ker(h^i)\text{ and
}
m_i:=n-p_i-q_i.
$$
Then $\sum m_i\leq n$ with equality if and only if $\mathfrak V$ is
exact, by Proposition \ref{pqm}. So, let $m(\mathfrak V):=\sum m_i$.

If $\mathfrak W$ is an expansion of $\mathfrak V$, then
$m(\mathfrak W)\geq m(\mathfrak V)$. Thus, to prove the lemma, we need
only show that, if $\mathfrak V$ is not exact, then there is an
expansion $\mathfrak W$ of $\mathfrak V$ with
$m(\mathfrak W)> m(\mathfrak V)$.

Thus, suppose $\mathfrak V$ is not exact. Let $i\in\mathbb Z$ for
which $\rk(h_i)+\rk(h^i)<n$. Let $W\subseteq V^{(i)}\oplus V^{(i+1)}$ be
a $n$-dimensional
subspace containing $\Im(h_i)\oplus\Im(h^i)$ such that the projection
maps $W\to V^{(i)}$ and $W\to V^{(i+1)}$ have images $\Ker(h^i)$
and $\Ker(h_i)$, respectively. This is possible:
$W=(\Im(h_i)\oplus\Im(h^i))+K$, where $K$ is the graph of an
isomorphism between a subspace of
$\Ker(h^i)$ complementary to $\Im(h_i)$ and a 
subspace of $\Ker(h_i)$ complementary to $\Im(h^i)$. In particular, it follows
that
$$
W\cap (V^{(i)}\oplus 0)=\Im(h_i)\oplus 0\text{ and }
W\cap (0\oplus V^{(i+1)})=0\oplus\Im(h^i).
$$
Thus, inserting $W$ between $V^{(i)}$ and $V^{(i+1)}$, we obtain a
  sequence of linked vector spaces $\mathfrak W$ expanding $\mathfrak
  V$ with
$$
m(\mathfrak W)=m(\mathfrak V)+(n-\rk(h_i)-\rk(h^i))>m(\mathfrak V).
$$
\end{proof}

\begin{prop}\label{pv}
Let $\mathfrak V:=(V^{(i)},h^i,h_i\,|\,i\in\z)$
be a sequence of linked vector
  spaces of dimension $n$. If $\mathfrak V$ is exact, then
  $\IP(\mathfrak V)$ is a connected, Cohen--Macaulay subscheme of
  $\IP^{1,1}(\mathfrak V)$ of pure dimension $n-1$ and bivariate
  Hilbert polynomial $h^0(\IP(\mathfrak V),\O_{\IP(\mathfrak V)}(i,j))=\binom{i+j+n-1}{n-1}$
  for $i,j>>0$. Conversely, if $\IP(\mathfrak
  V)$ has bivariate Hilbert polynomial $\binom{i+j+n-1}{n-1}$, then
  $\mathfrak V$ is exact.
\end{prop}

\begin{proof} The proof of the first statement follows step-by-step
the same proof given to \cite{E-Oss}, Thm.~4.3. p.~84. 
In particular, it follows from that proof that, if $\mathfrak V$ is exact, then
\begin{equation}\label{irredPV}
\IP(\mathfrak V)=\bigcup_{i\in S_{\mathfrak V}}\IP(\mathfrak V_i),
\end{equation}
where
$$
S_{\mathfrak
  V}:=\{i\in\z\,|\,V^{(i)}\neq\Ker(h^i)\oplus\Ker(h_{i-1})\}=
\{i\in\z\,|\, m_i>0\};
$$
see
\cite{E-Oss}, Rmk.~4.9, p.~89. Furthermore, it follows from Lemmas 4.4 and 4.8
in {\it loc.~cit.} that the $\IP(\mathfrak V_i)$ are irreducible of
dimension $n-1$ if $i\in S_{\mathfrak V}$. In other words,
\eqref{irredPV} is the expression of $\IP(\mathfrak V)$ as the union
of its irreducible components; there are $|S_{\mathfrak V}|$ of
them.

Conversely, let $\mathfrak W$ be an exact sequence of linked vector
spaces
expanding $\mathfrak V$. Then $\IP(\mathfrak V)\subseteq\IP(\mathfrak
W)$. Furthermore, if $\mathfrak V$ is not exact then
$S_{\mathfrak V} \subsetneqq S_{\mathfrak W}$, and thus there
is $i\in\z$ such that $\IP(\mathfrak W_i)\not\subseteq\IP( \mathfrak
V)$. So the bivariate Hilbert polynomial of $\IP(\mathfrak V)$ must be
different
from that of $\IP(\mathfrak W)$, thus different from
$\binom{i+j+n-1}{n-1}$ by the first statement of the proposition.
\end{proof}

\begin{thm}\label{thm:1} Let $r,d,\delta,\delta'$ be nonnegative
  integers, with $\delta|\delta'$.
\begin{enumerate}
\item For each $\mathfrak{g}\in G^{r}_{d,\delta}(X)$, if $\mathfrak g$
  is exact then the
  subscheme $\p(\mathfrak{g})\subseteq X^{(d)}$ is connected,
  Cohen--Macaulay, of pure dimension $r$ and bivariate Hilbert polynomial
$\binom{i+j+r}{r}$. Conversely, if $\IP(\g)$ has bivariate Hilbert polynomial
$\binom{i+j+r}{r}$, then $\g$ is exact.
\item For each $\mathfrak{g'}\in G^{r}_{d,\delta'}(X)$, letting
$\mathfrak g:=\rho_{\delta',\delta}(\mathfrak{g'})$, we have
$\p(\mathfrak{g})\subseteq\p(\mathfrak{g}')$. If $\mathfrak g$ is
exact, then so is $\mathfrak g'$, and equality holds. Conversely, if
equality holds and $\mathfrak g'$ is exact, then so is $\g$.
\end{enumerate}
\end{thm}

\begin{proof} Statement 1 follows from Proposition \ref{pv},
  since $\IP(\mathfrak V)$ and $\IP(\g)$ coincide, and $\mathfrak V$
  is exact if and only if $\g$ is, where $\mathfrak V$ is the sequence
  of linked vector spaces associated to $\g$.

The first assertion of the second statement follows
from the fact that $\mathfrak V'$ is
an expansion of $\mathfrak V$, where $\mathfrak V'$ is the sequence
  of linked vector spaces associated to $\g'$. The remaining
  assertions follow from Proposition~\ref{Cor1} and Statement~1.
\end{proof}

\begin{thm}\label{thm:flat-limit} Let $\pi\:\X\to B$ be a smoothing of $X$
  with singularity degree $\delta$. Let $\I$ be a torsion-free, rank-1
  sheaf on $\X/B$ and $V\subseteq\Gamma(X_\eta,I_\eta)$ a vector
  subspace. Let $\mathfrak g$ be the level-$\delta$ limit linear
  series on $X$ that arises from $\mathfrak g_\eta:=(I_\eta,V)$.
Then $\IP(V)$, viewed as a
subscheme of the fiber over $\eta$ of the
relative symmetric product $\X^{(d)}_B$, has
closure intersecting $X^{(d)}$ in
$\IP(\mathfrak g)$.
\end{thm}

\begin{proof} As in the proof to \cite{E-Oss}, Thm.~5.2, p.~90,
since $\g$ is exact, and thus $\IP(\mathfrak g)$ has the bivariate
Hilbert polynomial of the limit of $\IP(V)$, we need only show that
the
closure of $\IP(V)$ contains $\IP(\mathfrak g)$.

Recall the notation used in Section \ref{lls3}. Let $\I^{(i)}$ denote
the twists of $\I$, and for each $i\in\mathbb Z$, let
$\V^{(i)}$ be the subsheaf of $\pi_*\I^{(i)}$ consisting of the
sections that restrict to sections in $V$ on the generic
fiber. For each $i\in\mathbb Z$, consider on the product
$\X\times_B\IP(\V^{(i)})$ the
composition
$$
\O_{\IP(\V^{(i)})}(-1) \lra \wt\V^{(i)} \lra \I^{(i)}
$$
where the first map is the tautological map of $\IP(\V^{(i)})$ and the
second is the evaluation map, all sheaves
and maps
being viewed on the product
under the appropriate pullbacks. Taking its dual and twisting by
$\O_{\IP(\V^{(i)})}(-1)$, we obtain a map to $\O_{\IP(\V^{(i)})}$
whose image is the sheaf of ideals of a subscheme
$F_i\subseteq\X\times_B\IP(\V^{(i)})$. Moreover, $F_i$ is a flat
subscheme of relative length $d$ over
$\IP(\V^{(i)})-(\IP(\Ker(h^i))\cup\IP(\Ker(h_{i-1})))$, where
$(V^{(i)},h^i,h_i\,|\,i\in\mathbb Z)$ is the sequence of linked
vector spaces associated to $\g$. Thus we obtain a map
$$
\IP(\V^{(i)})-(\IP(\Ker(h^i))\cup\IP(\Ker(h_{i-1})))\lra \X^{(d)}_B
$$
whose image contains $\IP(V)$ and all points of $X^{(d)}$ of the
form $\div(s|_Y)+\div(s|_Z)$ for $s\in V^{(i)}-(\Ker(h^i)\cup\Ker(h_{i-1})$.
Since $\IP(\V^{(i)})$ is flat over $B$, it follows that the closure
of $\IP(V)$ in $\X^{(d)}_B$ contains all points of the above form.
As we let $i$ vary, we get that the closure of $\IP(V)$ contains
$\IP(\mathfrak g)$.
\end{proof}

\bibliographystyle{plain}

\begin{thebibliography}{1}

\bibitem{ACH}
Arbarello, E., Cornalba, M and Griffiths, P.
\newblock Geometry of algebraic curves, vol.2
\newblock {\em Grundlehren der Mathemaitschen Wissenchaften}, Springer--Verlag,
\newblock Vol. 168, 2011.

\bibitem{CEG}
Cumino, C., Esteves, E., Gatto, L.,
\newblock Limits of special Weiertrass points.
\newblock  {\em Int. Math. Res. Papers}, \textbf{2008}:1--65, (2008).

\bibitem{EH1}
Eisenbud, D., Harris, J.,
\newblock Limit linear series: Basic theory.
\newblock {\em Invent. Math.}, \textbf{85}:337--371, (1986).

\bibitem{EE2}
Esteves, E.,
\newblock Compactifying the relative Jacobian over families of reduced curves.
\newblock {\em Trans. of the Amer. Math. Soc.}, \textbf{353(8)}:3045--3095, (2001).

\bibitem{ENR2}
Eduardo, E., Nigro, A., Rizzo, P.,
\newblock Stable limit linear series.
\newblock {To appear.}

\bibitem{E-Oss}
Esteves, E., Osserman, B.,
\newblock Abel maps and limit linear series.
\newblock {\em Rend. Circ. Mat. Palermo}, \textbf{62(1)}:79--95, (2013).

\bibitem{EP}
Esteves, E., Pacini, M.,
\newblock Semistable modifications of families of curves and compactified Jacobians.
\newblock {\em Ark. Mat.}, \textbf{54}:55--83, (2016).

\bibitem{La}
Langton, S.,
\newblock Valuative criteria for families of vector bundles over
algebraic varieties
\newblock {\em Ann. Math.}, \textbf{101}:88--110, (1975).

\bibitem{Liu}
Liu, Fu,
\newblock Moduli of crude limit linear series.
\newblock {\em Int. Math. Res. Not.}, \textbf{2009(21)}:4032--4050, (2009).

\bibitem{Oss1}
Osserman, B.,
\newblock A limit linear series moduli scheme.
\newblock {\em Ann. Inst. Fourier (Grenoble)}, \textbf{56(4)}:1165--1205, (2006).

\bibitem{Oss1405}
Osserman, B.,
\newblock Limit linear series moduli stacks in higher rank
\newblock {\em Preprint}, (2014), (arXiv:1405.2937). 

\bibitem{Oss1406}
Osserman, B.,
\newblock Limit linear series for curves not of compact type
\newblock {\em Preprint}, (2014), (arXiv:1406.6699). 

\bibitem{Rizzo}
Rizzo, P.,
\newblock Level-delta and stable limit linear series on singular
curves.
\newblock {Tese de Doutorado, IMPA, Rio de Janeiro, 2013}.

\end{thebibliography}

\vskip0.5cm 

{\smallsc Instituto Nacional de Matem\'atica Pura e Aplicada, 
Estrada Dona Castorina 110, 22460--320 Rio de Janeiro RJ, Brazil}

{\smallsl E-mail address: \small\verb?esteves@impa.br?}

\vskip0.5cm

{\smallsc Universidade Federal Fluminense, Instituto de Matem\'{a}tica
  e Estat\'\i stica,
Rua M\'{a}rio Santos Braga, s/n, 24020--140 Niter\'{o}i RJ, 
Brazil}

{\smallsl E-mail address: \small\verb?antonio.nigro@gmail.com?}

\vskip0.5cm

{\smallsc Universidad de Antioquia, Instituto de Matem\'aticas, Calle
  67 No 53-108, Medell\'\i n, Colombia}

{\smallsl E-mail address: \small\verb?pedro.hernandez@udea.edu.co?}

\end{document}